\documentclass[11p,reqno]{amsart}
\usepackage{amsmath}
\usepackage{amsfonts}
\usepackage{amssymb}
\usepackage{graphicx}
\usepackage{color}
\newtheorem{theorem}{Theorem}[section]\newtheorem{thm}[theorem]{Theorem}
\newtheorem*{theorem*}{Theorem}
\newtheorem{lemma}{Lemma}[section]

\newtheorem{prop}{Proposition}[section]
\newtheorem{definition}[theorem]{Definition}

\def \b {\beta}

\def\k{\kappa}

\def\a{\alpha}
\def\w{\omega}
\def\l{\lambda}
\def\ol{\overline}

\def\e{\epsilon}
\def\p{\partial}

\def\R{\mathbb{R}}

\def\vp{\varphi}

\def\tr{\operatorname{tr}}

\def\USC{\operatorname{USC}}
\def\LSC{\operatorname{LSC}}

\DeclareMathOperator\erf{erf}

\numberwithin{equation}{section}

\begin{document}

\title[Modulus of Continuity Estimates ]{Modulus of Continuity Estimates for Fully Nonlinear Parabolic Equations}

\author{Xiaolong Li}
\address{Department of Mathematics, University of California, Irvine, Irvine, CA 92697, USA}
\email{xiaolol1@uci.edu}






\maketitle

\begin{abstract}
    We prove that the moduli of continuity of viscosity solutions to fully nonlinear parabolic partial differential equations
    are viscosity subsolutions of suitable parabolic equations of one space variable. As applications, we obtain sharp Lipschitz bounds and gradient estimates for fully nonlinear parabolic equations with bounded initial data, via comparison with one-dimensional solutions. 
    This work extends multiple results of Andrews and Clutterbuck for quasilinear equations to fully nonlinear equations. 
\end{abstract}


\section{Introduction} 

Given a function $u: \Omega \subset \R^n \to \R$,
any function $f$ satisfying 
$$|u(x)-u(y)| \leq 2 f\left(\frac{|x-y|}{2}\right)$$
for all $x, y \in \Omega$ is called \textit{a modulus of continuity} of $u$. The (optimal) \textit{modulus of continuity} $\w$ of $u$ is defined by 
\begin{equation*}
    \w(s)=\sup \left\{\frac{u(x)-u(y)}{2} : x, y \in \Omega,  |x-y|=2s \right\}. 
\end{equation*}
The definitions here are consistent with \cite{Andrewssurvey12}\cite{Andrewssurvey15}\cite{AC09a}\cite{AC09}, but differ from the usual ones by the factors of $2$, which are included for convenience and nice statements of results. For instance, if $\vp_0(s)$ is an odd function defined on a symmetric interval and it is positive and concave for positive $s$,
then its modulus of continuity is exactly $\w(s)=\vp_0(s)$ for all $s \geq 0$. 
If we then evolve $\vp_0$ as initial data by a quasilinear parabolic equation of one space variable of the form $\vp_t =\a(\vp')\vp''$, then it is easy to see that the solution $\vp(\cdot,t)$ remains odd, concave and positive for positive $s$. 
Therefore the modulus of continuity $\w(s,t)$ of the solution $\vp(s,t)$ is exactly $\w(s,t)=\vp(s,t)$. 

Inspired by the above example, Andrews and Clutterbuck \cite{AC09a} first observed that the modulus of continuity of a regular ($C^2$ in space variable and $C^1$ in time) periodic solution of the quasilinear parabolic equation $\vp_t =\a(\vp')\vp''$ is a subsolution of the same equation. Their proof is inspired by an argument of doubling variables used by Kru\v{z}kov \cite{Kruzhkov67} for linear parabolic equations of one space variable. 
In \cite{AC09}, they managed to generalize this to higher dimensions, showing that for a wide class of quasilinear parabolic equations including the anisotropic mean curvature flows, the modulus of continuity of a periodic regular solution is a subsolution of an associated one-dimensional equation. 
Moreover, the results are sharp in the sense that initial data close to a square-wave function of one of the variables will give equality in the limit of lattices with large period. 
Equations with Dirichlet or Neumann boundary conditions were also treated in \cite{AC09} with convexity assumptions on the domain. As applications, they obtained time-interior gradient estimates (more precisely,
estimates on the gradient for positive times which do not depend on the initial gradient, but only on the oscillation of the initial data) for solutions of quasilinear parabolic equations with gradient-dependent coefficients, under the weakest possible assumptions on the coefficients. 
Such estimates have not yet been accomplished using direct estimates on the gradient except in some special cases \cite{Clutterbuck07}. 
Later on, the modulus of continuity estimates were extended to quasilinear isotropic equations on Riemannian manifolds in \cite{AC13}\cite{AN12}\cite{Ni13}
and to viscosity solutions in \cite{Li16}\cite{LW17}. 
At last, we would like to mention that the modulus of continuity and its variants have found remarkable applications in proving sharp lower bounds for lower eigenvalues of the Laplacian in \cite{AC11}\cite{AC13}\cite{AN12}\cite{DSW18}\cite{HW17}\cite{Ni13}\cite{SWW19}, the $p$-Laplacian in \cite{Andrewssurvey15} and the weighted $p$-Laplacian in \cite{LW19eigenvalue}\cite{LW19eigenvalue2}.
We refer the reader to the nice surveys by Professor Andrews \cite{Andrewssurvey12}\cite{Andrewssurvey15}, where these ideas were further explained, various applications are discussed, and connections to other problems in geometric analysis are made.




The purpose of the present paper is to extend the above-mentioned modulus of continuity estimates for linear and quasilinear parabolic equations to fully nonlinear parabolic equations. 
We will show that the moduli of continuity of viscosity solutions to fully nonlinear parabolic equations are viscosity subsolutions of suitable parabolic equations of one space variable (see Theorems \ref{thm MC main}, \ref{thm MC Neumann} and \ref{thm MC mfd} for precise statements). In contrast to the quasilinear case, the one-dimensional equations are determined by a structure condition (see \eqref{eq 1.3 structure} below) that we introduce on the fully nonlinear operator, rather than canonically associated.
The structure condition specifies all the properties that need to be satisfied by the one-dimensional operators, and we then provide numerous examples (see Propositions \ref{prop identify 1D}, \ref{prop 1D identify} and \ref{prop identify 1D mfd}) to illustrate how to choose a natural one-dimensional operator for the given fully nonlinear operator.

Consider parabolic partial differential equations of the form 
\begin{equation*}
    u_t +F(t,x,u,Du,D^2 u)=0,
\end{equation*}
on a domain $\Omega \subset \R^n$, where $F:[0,T] \times \Omega \times \R \times \R^n \times S(n) \to  \R$,  $S(n)$ is the set of symmetric $n \times n$ matrices, $u_t$ is the time derivative of $u$, $Du$ is the gradient of $u$, and $D^2u$ is the Hessian of $u$. 
Throughout the paper we assume that $F$ is degenerate elliptic, i.e., 
\begin{equation*}
    F(t,x,r,p,X) \leq F(t,x,r,p,Y) \text{ whenever }  Y\leq X. 
\end{equation*}
We also assume $F$ is a continuous function of its arguments, so the basic theory of viscosity solutions in \cite{CIL92} applies.

A fundamental difficulty in proving the modulus of continuity estimates for fully nonlinear equations is that it is not clear how to identify the one-dimensional operators. In the quasilinear isotropic case considered in \cite{AC09a} and \cite{AC09}, the one-dimensional operator is obtained by plugging into a function that depends only on one of the variables. For example, the associated one-dimensional operator of the Laplacian ($F=-\tr(X)$) is the one-dimensional Laplacian ($f(\vp)=-\vp''$).
However, if the operator has lower order terms depending on $x$ (say $F=-\tr(X) +h(x)$ for some function $h(x)$) or its coefficients depending on $x$ (say $F=-\tr(AX)$ for $A=(a_{ij}(x)$), one certainly cannot identify its one-dimensional operator by plugging a solution of one space variable, 
not to mention for more general fully nonlinear operators. 
To overcome this difficulty, we introduce a structure condition on $F$ as follows. 

Given an operator $F(t,x,r,p,X)$, let 
$f(t,s,\vp,\vp',\vp'')$ be a one-dimensional operator (also degenerate elliptic and continuous of its arguments) such that
\begin{equation}\label{eq 1.3 structure} \nonumber
    \begin{cases}
    & F\left(t,y,r,\vp'\frac{x-y}{|x-y|}, Y\right) -F\left(t,x,v, \vp'\frac{x-y}{|x-y|}, X\right) \leq -2f(t, s, \vp,\vp',\vp'') \\ \nonumber
     \tag{SC} & \text{ for all } 
    x, y\in \Omega \text{ with } |x-y|=2s>0, v,r \in \R \text{ with } v-r=2\vp >0, \\ \nonumber
    & \text{ and } X, Y \in S(n) \text{ satisfying }
     \begin{pmatrix}
    X & 0 \\ 0 & -Y 
    \end{pmatrix}
    \leq D^2_{x,y} \left(2\vp\left(\frac{|x-y|}{2}, t\right) \right),
    \end{cases}
\end{equation}
where all derivatives of $\vp$ are evaluated at $s=\frac{|x-y|}{2}$ and $D^2_{x,y}$ means taking the Hessian with respect to all spatial variables.

The above structure condition is inspired by the structure condition in \cite[(3.13) and (3.14)]{CIL92}, under which the comparison principle is proved for a large class of operators. Also, it will be clear in the proofs that such a structure condition is exactly what we need. 
Indeed, the one-dimensional operator $f$ is in some sense a modulus of continuity of $F$, as it measures the change of $F$ when its arguments change under the constrains specified above. At last, such one-dimensional operators exist for very general $F$, and 
we will discuss how to choose $f$ for a large class of $F$ in Proposition \ref{prop identify 1D} after stating our main result below.  

\begin{thm}\label{thm MC main}
Suppose $u:\R^n \times [0,T) \to \R$ is a spatially periodic viscosity solution of 
\begin{equation}\label{eq 1.1 fully nonlinear}
    u_t +F(t,x,u,Du,D^2 u)=0.
\end{equation}
Let $f(t, s,\vp, \vp',\vp'')$ be a one-dimensional operator satisfying the structure condition \eqref{eq 1.3 structure}.
Then the modulus of continuity 
\begin{equation*}
    \w(s,t):=\sup\left\{\frac{u(x,t)-u(y,t)}{2} : |x-y|=2s \right\}
\end{equation*}
of $u$ is a viscosity subsolution of the one-dimensional equation 
\begin{equation}\label{eq 1.4 1D}
    \w_t + f(t, s,\w, \w', \w'')=0
\end{equation}
on $(0,\infty)\times (0,T)$. 
\end{thm}

The same conclusion as in Theorem \ref{thm MC main} holds if $u$ solves \eqref{eq 1.1 fully nonlinear} on a bounded convex domain $\Omega \subset \R^n$ with smooth boundary and satisfies the Neumann boundary condition. This will be proved in Section \ref{Sec Neumann}. The Dirichlet boundary condition can also be handled with more restrictions on both $F$ and $\w$. 

The structure condition \eqref{eq 1.3 structure} specifies all the requirements for the one-dimensional operator, but does not provide any choice of it directly. 
So we provide in the next proposition some natural choices of the one-dimensional operators according to the given $F$. More examples will be given in Section \ref{Sec 1D}.


\begin{prop}\label{prop identify 1D}
The following pairs of operators $F$ and $f$ satisfy the structure condition \eqref{eq 1.3 structure}. 
\begin{enumerate}
\item $F$ is the linear elliptic operator given by \begin{equation*}
    F(x,r,p,X) =-\tr(X) -\langle W(x), p \rangle -V r -h(x),
\end{equation*}
where $W$ is a bounded continuous vector field, $h$ is a bounded continuous function, and $V \in \R$. 
\begin{equation*}
    f(\vp, \vp',\vp'')=-\vp'' -K |\vp'| - V \vp -\w_h,
\end{equation*}
where $K=\sup_{x}|W(x)|$ and $w_h$ is a modulus of continuity of $h$. 
    \item $F$ is the quasilinear isotropic operator given by 
    \begin{equation*}\label{}
        F(p,X)=-\tr \left[ \left(\a(|p|)\frac{p \otimes p}{|p|^2} +\b(|p|) \left(I- \frac{p \otimes p}{|p|^2}\right) \right) X \right],
\end{equation*}
    where $\a$ and $\b$ are nonnegative functions, $I$ denotes the identity matrix and $p \otimes q$ denotes the matrix whose $(i,j)$ entry is $p_iq_j$.
    \begin{equation*}
        f(\vp,\vp'')=-\a(|\vp'|)\vp''.
    \end{equation*}
     \item $F(X)$ is uniformly elliptic with ellipticity constants $0<\l \leq \Lambda < \infty$, i,e., 
     \begin{equation*}
         -\Lambda \tr(Z) \leq F(X+Z)-F(X) \leq -\l \tr(Z) \text{ for } Z \geq 0.
     \end{equation*}
     \begin{equation*}
         f(\vp'')=-\l \vp''. 
     \end{equation*}
     \item $F(t,r,p,X)$ (independent of $x$) is proper, i.e., 
     \begin{equation*}
         F(t,r,p,X) \leq F(t,v,p,Y) \text{ whenever } r\leq v, Y \leq X.  
     \end{equation*}
     \begin{equation*}
         f \equiv 0.
     \end{equation*}
     
     \item $F(t,x,r,p,X)$ satisfies that for all $x,y\in \Omega$, $t\in [0,T]$, $v,r \in \R$ and $X,Y,Z \in S(n)$,
     \begin{align*}
         |F(t,y,r,p,X) -F(t,x,r,p,X)| & \leq L |x-y|, \\
          F(t,x,r,p,X) -F(t,x,v,p,X) & \leq K (v-r) \text{ for } v\geq r,\\
          F(t,x,r,p,X+Z) -F(t,x,r,p,X) & \leq -\lambda(|p|,t) \tr(Z) \text{ for } Z\geq 0, 
     \end{align*}
     where $L, K$ are positive constants and $\l(s,t)$ is a nonnegative function.
     \begin{equation*}
         f(s,\vp,\vp',\vp'')=-\l(|\vp'|,t)\vp'' -K \vp -L s .
     \end{equation*}
\end{enumerate}
\end{prop}

Part (1) of Proposition \ref{prop identify 1D} says that $f$ is linear if $F$ is linear. It implies exponentially gradient bounds for linear equations because the one-dimensional can be solved explicitly (choosing $\w_h=\frac{1}{2}\sup_{x}|h(x)|$); see \cite{Balch20}. 

Applying Theorem \ref{thm MC main} to the operators in Part (2) covers the main results obtained by Andrews and Clutterbuck in \cite{AC09a} and \cite{AC09} for quasilinear equations. Note that quasilinear isotropic operators in Part (2) include many greatly studied elliptic operators such as the Laplacian (with $\a=\b=1$), the $p$-Laplacian (with $\a(s)=(p-1)|s|^{p-1}$ and $\b(s)=|s|^{p-1}$), and graphical mean curvature operator (with $\a(s)=\frac{1}{1+s^2}$ and $\b=1$). 

Part (3) implies that $f$ can be chosen to be linear as long as $F$ is uniformly elliptic. In particular, if $F$ is the Pucci's extremal operators $-M^{\pm}_{\l,\Lambda}(D^2 u)$ with ellipticity constants $0<\l \leq \Lambda < \infty$, then $f$ can be chosen to be $f(\vp'')=-\l \vp''$. 
As we will see in Section \ref{Sec Lip}, uniformly elliptic fully nonlinear operators behave just like uniformly elliptic linear operators in terms of Lipschitz bounds and time-interior gradient estimates. 

Part (4) implies that if $F$ is independent of $x$ and proper, then the parabolic equation \eqref{eq 1.1 fully nonlinear} preserves any initial modulus of continuity for spatially periodic solutions (or with Neumann boundary condition), i.e., if $\vp(s)$ is a modulus of continuity for $u(\cdot, 0)$, then it is also a modulus of continuity for $u(x,t)$ whenever the solution exists. 

Part (5) addresses the more general situation where the ellipticity constants may depend on the norm of the gradient and  and $F(t,x,v,p,X)$ is Lipschitz in $x$ and $v$. 
Examples of such operators include the $p$-Laplacian $\Delta_p$, $-\l(|Du|) \Delta u$ and $-|Du|^\gamma M^{\pm}_{\l,\Lambda}(D^2 u)$. 
In these cases, the one-dimensional equations can be chosen to be quasilinear.

We conclude this section by mentioning various extensions and applications of Theorem \ref{thm MC main}, as well as discussing the organization of this paper. 
In Section 2, we recall the definitions of viscosity solutions for parabolic equations and state the parabolic maximum principle for semicontinuous functions, which is the key tool that we use in this paper. 
We then present the proofs of Theorem \ref{thm MC main} and Proposition \ref{prop identify 1D} in Sections 3 and 4, respectively. 
In Section 5, we prove modulus of continuity estimates when the Neumann boundary condition is imposed. 
Section 6 is devoted to proving Lipschitz bounds and gradient estimates for solutions to parabolic equations with bounded initial data, as applications of the modulus of continuity estimates. 
In Section 7, we extend the modulus of continuity estimates to fully nonlinear parabolic equations on Riemannian manifolds. 
In Section 8, we study the effects of curvatures on the one-dimensional operators. 

\section{Preliminaries on Viscosity Solutions} 
In this section, we collect some basics on the theory of viscosity solutions that will be needed in the sequel. The reader is encouraged to consult \cite{CIL92} for a self-contained exposition of the basic theory of viscosity solutions.

Let $\Omega\subset \R^n$ be an open set and $O$ be an open subset of $\Omega \times (0,T)$.  The following notations are useful:
\begin{align*}
    \USC(O) & =\left\{u:O \to \R |\  u \mbox{  is upper semicontinuous }\right\},\\
    \LSC(O) & =\left\{u:O \to \R |\  u \mbox{  is lower semicontinuous }\right\}.
\end{align*}

Next we introduce the notion of parabolic semijets. 
\begin{definition}
\begin{enumerate}
    \item For a function $u\in \USC(O)$, the second order parabolic superjet of $u$ at a point $(x_0,t_0) \in O$ is defined by 
\begin{align*}
    \mathcal{P}^{2,+}u(x_0,t_0)  =\{& \left(\vp_t(x_0,t_0), D\vp(x_0,t_0), D^2 \vp (x_0,t_0) \right) : \\
    & \vp \in C^{\infty}(O) \text{ such that } u-\vp \text{ attains a local maximum at } (x_0,t_0)\}. 
\end{align*}
\item For $u\in \LSC(O)$, the second order parabolic subjet of $u$ at $(x_0,t_0) \in O$ is defined by 
\begin{align*}
    \mathcal{P}^{2,-}u(x_0,t_0) = - \mathcal{P}^{2,+}(-u)(x_0,t_0) .
\end{align*}
\item We also define the closures of $\mathcal{P}^{2,+}u(x_0,t_0) $ and $\mathcal{P}^{2,-}u(x_0,t_0) $ by 
\begin{align*}
\overline{\mathcal{P}}^{2,+}u(z_0)
=\{ & (\tau,p,X)\in \R \times \R^n \times S(n) | 
\mbox{  there is a sequence  } (z_j,\tau_j,p_j,X_j) \\
&\mbox{  such that  } (\tau_j,p_j ,X_j)\in \mathcal{P}^{2,+}u(z_j) \\
&\mbox{  and  } (z_j,u(z_j),\tau_j,p_j,X_j) \to (z_0,u(z_0),\tau, p ,X) \mbox{  as  } j\to \infty \}.  \\
\overline{\mathcal{P}}^{2,-}u(z_0)&=-\overline{\mathcal{P}}^{2,+}(-u)(z_0). 
\end{align*}
\end{enumerate}
\end{definition}

Now we can give the definition of a viscosity solution for the general equation
\begin{equation} \label{nonsingular equation}
u_t+F(t, x, u, D u, D^2 u)=0,
\end{equation} 
where $F$ is assumed to be degenerate elliptic and continuous in its arguments. 

\begin{definition}\label{def viscosity}
\begin{enumerate}
    \item A function $u \in \USC(O)$ is a viscosity subsolution of \eqref{nonsingular equation}
in $O$ if for all $(x,t)\in O$ and $(t,p,X)\in \mathcal{P}^{2,+}u(x,t) $,
\begin{align*}
\tau + F(t,x,u(x,t),p, X) \leq 0. 
\end{align*}
\item  A function $u \in \LSC(O)$ is a viscosity supersolution of \eqref{nonsingular equation}
in $O$ if for any $(x,t) \in O$ and $(t,p,X)\in \mathcal{P}^{2,-}u(x,t) $,
\begin{align*}
\tau + F(t,x,u(x,t),p, X) \geq 0.
\end{align*}
\item A viscosity solution of \eqref{nonsingular equation} in $O$ is a continuous function which is both a viscosity subsolution and a viscosity supersolution of \eqref{nonsingular equation} in $O$. 
\end{enumerate}
\end{definition}

We state the parabolic maximum principle for semicontinuous functions (see \cite[Theorem 8.3]{CIL92}), which is the key tool in proving the modulus of continuity estimates. 
\begin{thm}\label{thm max prin semicts}
Let $u_i \in  \USC(O_i \times (0, T))$ for $i = 1, ... , k$, where $O_i$ is a locally compact subset of $\R^{n_i}$. Let $\vp(x_1, \cdots, x_k, t)$ be a smooth function defined on an open neighborhood of $O_1 \times \cdots \times O_k \times (0,T)$. 
Suppose the function 
\begin{equation*}
    w(x_1, \cdots, x_k, t) :=u_1(x_1, t) + \cdots u_k(x_k, t) -\vp(x_1, \cdots, x_k, t)
\end{equation*}
attains a maximum at $(\hat{x}_1, \cdots, \hat{x}_k, \hat{t}) $ on $O_1 \times \cdots \times O_k \times (0,T)$. Assume further that there is an $r >0$  such that for every $\eta >0$ there is a $C>0$ such that for $i=1, \cdots, k$
\begin{align*}
& b_i \leq C  \mbox{  whenever  } (b_i,q_i,X_i) \in \mathcal{P}^{2,+}u_i(t,x_i) , \\
& |x_i- \hat{x}_i|+|t-\hat{t}| \leq r \mbox{  and  } |u_i(x_i,t)|+|q_i| +\|X_i\| \leq \eta.
\end{align*}
Then for each $\lambda>0$, there are $X_i \in S(n)$ such that
\begin{align*}
& (b_i,D_{x_i}\vp(\hat{x}_1, \cdots, \hat{x}_k, \hat{t}),X_i)  \in \ol{\mathcal{P}}^{2,+}u_i(\hat{t},\hat{x}_i),\\
&  -\left(\frac 1 \lambda +\left\|M\right\| \right)I \leq
    \begin{pmatrix}
   X_1 & \cdots & 0 \\
   \vdots & \ddots & \vdots \\
   0 & \cdots & X_k
   \end{pmatrix}
   \leq M+\lambda M^2,  \\
& b_1 + \cdots + b_k =\vp_t(\hat{x}_1, \cdots, \hat{x}_k,\hat{t}),
\end{align*}
where $M=D^2\vp(\hat{x}_1, \cdots, \hat{x}_k,\hat{t})$ and $\|M\|=\sup\{M(v,v) : \|v\|=1 \}.$
\end{thm}

\section{Modulus of Continuity Estimates: Proof of Theorem \ref{thm MC main}}

In this section, we prove Theorem \ref{thm MC main}. The proof relies on the maximum principle for semicontinuous functions. 

\begin{proof}[Proof of Theorem \ref{thm MC main}]
By definitions of viscosity solutions (see Definition \ref{def viscosity}), we need to show that for every smooth function $\vp$ touching $\w$ from above at $(s_0,t_0)\in (0,\infty)\times (0,T)$, in the sense that 
\begin{equation*}
    \begin{cases}
    \vp(s,t) \geq w(s,t) \text{ for all $(s,t)$ near $(s_0,t_0)$}, \\
    \vp(s_0,t_0)  =\w(s_0,t_0), &
    \end{cases}
\end{equation*}
it holds at the point $(s_0,t_0)$ that 
\begin{equation*}
    \vp_t + f(t,s,\vp,\vp',\vp'') \leq 0.
\end{equation*}
By the definition of $\w(s,t)$, we have that 
\begin{equation*}
    u(x,t)-u(y,t) \leq 2\w\left(\frac{|x-y|}{2}, t \right) \leq 2\vp\left(\frac{|x-y|}{2}, t \right) 
\end{equation*}
for all points $x,y \in \R^n$ with $|x-y|$ close to $2s_0$ and $t$ close to $t_0$. 
Since $u$ is spatially periodic, there exist $x_0,y_0 \in \R^n$ with $|x_0-y_0|=2s_0$ such that 
\begin{equation*}
    u(x_0,t_0)-u(y_0,t_0) =2\w\left(\frac{|x_0-y_0|}{2}, t_0 \right) = 2\vp\left(\frac{|x_0-y_0|}{2}, t_0 \right).
\end{equation*}
In other words, the function 
$$Z(x,y,t) :=u(x,t)-u(y,t)- 2\vp\left(\frac{|x-y|}{2}, t \right)$$
attains a local maximum zero at $(x_0,y_0,t_0)$. Now we can apply the parabolic maximum principle for semicontinuous (see Theorem \ref{thm max prin semicts}) functions to conclude that for each $\l>0$, there exist $b_1,b_2 \in \R$ and $X, Y \in S(n)$ such that 
\begin{align*}
    (b_1, \vp' e, X) & \in \mathcal{\overline{P}}^{2,+}u(x_0,t_0), \\
    (-b_2, \vp' e, Y) & \in \mathcal{\overline{P}}^{2,-}u(y_0,t_0),\\
    b_1+b_2 & = 2\vp_t, \\
    \begin{pmatrix}
X & 0 \\
0 & -Y
\end{pmatrix} & \leq M+\l M^2,
\end{align*}
where $e=\frac{x_0-y_0}{|x_0-y_0|}$, $M=D^2_{x,y}\left( 2\vp \left(\frac{|x-y|}{2} \right),t \right)$, and all derivatives of $\vp$ are evaluated at $(s_0,t_0)$ here and in the rest of the proof. 

Since $u$ is a viscosity solution of \eqref{eq 1.1 fully nonlinear}, we have 
\begin{align*}
    b_1 +F(t_0,x_0,u(x_0,t_0), \vp' e, X) \leq 0, \\
    -b_2 +F(t_0,y_0,u(y_0,t_0),\vp' e, Y) \geq 0.
\end{align*}
Therefore, we obtain by letting $\l \to 0^+$ that
\begin{align*}
    2\vp_t=b_1+b_2 & \leq F(t_0,y_0,u(y_0,t_0),\vp' e, Y) -F(t_0,x_0,u(x_0,t_0), \vp' e, X) \\
    & \leq -2f(t_0,s_0,\vp,\vp',\vp''),
\end{align*}
where we have used \eqref{eq 1.3 structure} in the last inequality. The proof is complete. 
\end{proof}

\section{Identifying One-dimensional Operators}\label{Sec 1D}

In this section, we first prove Proposition \ref{prop identify 1D}, and then provide more examples to illustrate how to choose $f$ for $F$ to make the best use of Theorem \ref{thm MC main}. 

The following lemma will be useful. 
\begin{lemma}\label{lemma}
Suppose that $X,Y \in S(n)$ satisfy 
\begin{equation*}
  \begin{pmatrix}
    X & 0 \\ 0 & -Y 
    \end{pmatrix}
    \leq D^2_{x,y} \left(2\vp\left(\frac{|x-y|}{2}, t\right) \right).
\end{equation*}
Then we have $X\leq Y$ and 
\begin{equation*}
    \tr(X-Y) \leq 2\vp''\left(\frac{|x-y|}{2}, t\right) .
\end{equation*}
\end{lemma}
\begin{proof}[Proof of Lemma \ref{lemma}]
Note that the Hessian of $2\vp\left(\frac{|x-y|}{2},t \right)$ has the form $$\begin{pmatrix} P & -P \\
-P & P\end{pmatrix},$$
where $P=2D^2_x\vp\left(\frac{|x-y|}{2},t \right)$. 
Then $X \leq Y$ follows from the fact that the matrix $\begin{pmatrix} P & -P \\
-P & P\end{pmatrix}$ annihilates vectors of the form $\begin{pmatrix} x  \\
x\end{pmatrix}$. 
For any matrix $C$ such that the matrix $\begin{pmatrix} I & C \\
-C & I\end{pmatrix}$ is positive semidefinite, we have 
\begin{eqnarray*}
    \tr(X-Y) &=& \tr \left(\begin{pmatrix} I & C \\
C & I\end{pmatrix} \begin{pmatrix}
    X & 0 \\ 0 & -Y 
    \end{pmatrix} \right) \\
&\leq& \tr \left(\begin{pmatrix} I & C \\
C & I\end{pmatrix} \begin{pmatrix}
    P & -P \\ -P & P 
    \end{pmatrix} \right)  \\
&=& 2 \tr\left( (I-C)P \right).
\end{eqnarray*}
Choosing $C=I-2e \otimes e$ with $e=\frac{x-y}{|x-y|}$ produces 
$$\tr\left( (I-C)P \right) =2 P(e,e)=\vp''\left(\frac{|x-y|}{2},t\right).$$ Thus we have the desired estimate.  
\end{proof}

\begin{proof}[Proof of Proposition \ref{prop identify 1D} ]
(1). For simplicity, we write $e=\frac{x-y}{|x-y|}$. 
For any $x,y \in \R^n $ with $|x-y|=2s$ and $r,v \in \R$ with $v-r=2\vp$,  we have 
\begin{eqnarray*}
   && F(y,r,\vp' e ,Y)-F(x,v,\vp' e, X) \\
   &=& -\tr(Y) -\langle W(y), \vp' e \rangle -Vr-h(y) +\tr(X) +\langle W(x), \vp' e \rangle +V v +h(x) \\
   &= & \tr(X-Y) +\langle W(x)-W(y), \vp' e\rangle +V(v-r) +h(x) -h(y) \\
   &\leq & 2\vp'' +2\|W\|_{L^\infty} |\vp'| +2V\vp+2w_h(s),
\end{eqnarray*}
where we have used Lemma \ref{lemma} and the assumption that $\w_h$ is a modulus of continuity of $h$. Thus $f(\vp,\vp',\vp'')=-\vp'' -\|W\|_{L^\infty} |\vp'| -V\vp -w_h$ satisfies \eqref{eq 1.3 structure}. 

(2). Since the quasi-linear isotropic operator $F$ 
is invariant under rotations, we may choose an orthonormal basis $\{e_i\}_{i=1}^n$ with $e_1=\frac{x-y}{|x-y|}$ to simplify the calculations. 
With respect this basis, we have 
$$D^2_{x,y} \left(2\vp\left(\frac{|x-y|}{2}, t\right) \right) =\begin{pmatrix} P & -P \\
-P & P\end{pmatrix},$$
where 
\begin{equation*}
    P=\begin{pmatrix} 
    \frac 1 2 \vp''& 0 & \cdots & 0\\
0 & \frac{\vp'}{|x-y|} & \cdots & 0 \\
0 & 0 & \cdots & \frac{\vp'}{|x-y|}
    \end{pmatrix}.
\end{equation*}
Let
$$A=\begin{pmatrix}
\a(|\vp'|) & 0 & \cdots & 0\\
0 & \b(|\vp'|) & \cdots & 0 \\
0 & 0 & \cdots & \b(|\vp'|)
\end{pmatrix}$$
and 
$$C=\begin{pmatrix}
-\a(|\vp'|) & 0 & \cdots & 0\\
0 & \b(|\vp'|) & \cdots & 0 \\
0 & 0 & \cdots & \b(|\vp'|)
\end{pmatrix}.$$
It's easy to see that 
$\begin{pmatrix}
A & C \\
C & A
\end{pmatrix}$ is a positive semidefinite matrix.
Therefore, we have 
\begin{eqnarray*}
&& F(\vp e_1, Y) -F(\vp e_1, X) \\
&=& -\tr (AY)+\tr(AX) \\
&=& \tr\left[ \begin{pmatrix} A & C \\
C & A\end{pmatrix} \begin{pmatrix} X & 0 \\
0 & -Y\end{pmatrix}\right]  \\
&\leq&  \tr\left[ \begin{pmatrix} A & C \\
C & A\end{pmatrix} \begin{pmatrix} P & -P \\
-P & P\end{pmatrix} \right] \\
&=& 2\tr\left[ (A-C)P \right] \\
&=&  2 \a(|\vp'|)\vp''
\end{eqnarray*}
Thus, $f=-\a(|\vp'|)\vp''$ satisfies \eqref{eq 1.3 structure}. 

(3). This is a special case of Part (5). By Lemma \ref{lemma}, we have 
\begin{equation*}
    F(Y)-F(X) \leq -\l \tr(Y-X) \leq 2 \l \vp''.
\end{equation*}
So one can take $f(\vp'')=-\l \vp''$. 

For simplicity, we write $e=\frac{x-y}{|x-y|}$ in (4) and (5). 

(4) By Lemma \ref{lemma}, we have $X\leq Y$. In view of $v-r=2\vp>0$, properness of $F$ implies that 
\begin{equation*}
    F(t,r, \vp' e, Y) -F(t,v,\vp' e, X) \leq 0.
\end{equation*}
It follows that $f \equiv 0$ satisfies \eqref{eq 1.3 structure}. 

(5). Using the assumptions on $F$, we estimate that 
\begin{eqnarray*}
&& F(t,y,r, \vp' e, Y) -F(t,x, v,\vp' e, X) \\
&\leq& L|x-y| + F(t,x,r, \vp' e, Y) -F(t,x,v,\vp' e, X) \\
&\leq& L|x-y| +K(v-r) + F(t,x,v, \vp' e, Y) -F(t,x,v,\vp' e, X)\\
& \leq& L|x-y| +K(v-r) + \lambda(|\vp'|,t) \tr (X-Y) \\
&=& 2Ls + 2K\vp+ 2\lambda(|\vp'|,t) \vp'',
\end{eqnarray*}
where we used $|x-y|=2s$, $v-r=2\vp$ and Lemma \ref{lemma} in the last inequality. 
Thus $f(s,\vp,\vp',\vp'')=-\l(|\vp'|,t)\vp''-K \vp -Ls$ satisfies \eqref{eq 1.3 structure}. 

\end{proof}


Next, we provide a few more examples. 
\begin{prop}\label{prop 1D identify}
The following operators $F$ and $f$ satisfy the structure condition \eqref{eq 1.3 structure}. 
\begin{enumerate}
\item 
    \begin{equation*}\label{}
        F(p,X)=-\tr \left[ \left(I- \frac{p \otimes p}{1+|p|^2}\right) X \right],
\end{equation*}
    \begin{equation*}
        f=-\frac{\vp''}{1+(\vp')^2}.
    \end{equation*}
    \item
    \begin{equation*}\label{}
        F(p,X)=-\tr( A(p,t) X ),
\end{equation*}
where $A(p,t)=a_{ij}(p,t)$ and there exists a continuous function $\a(R,t)$ such that 
\begin{equation*}
0< \a(R,t) \leq R^2 \inf_{|p|=R, (v \cdot p) \neq 0} \frac{v^T A(p,t)v}{(v \cdot p)^2}
\end{equation*}
    \begin{equation*}
        f=-\a(\vp',t)\vp''.
    \end{equation*}
\end{enumerate}
\end{prop}

Theorem \ref{thm MC main} covers \cite[Theorem 2.1]{AC09} with the operators in part (1) and  \cite[Theorem 3.1]{AC09} with the operators in part (2). 

\begin{proof}[Proof of Proposition \ref{prop 1D identify}]
(1). This is a special case of part (1) in Proposition \ref{prop identify 1D} with $\a=\frac{1}{1+|p|^2}$ and $\b=1$. 

(2). The assumption implies $A(p,t) \geq \a(|p|,t) I$. So for $Z\geq 0$, we have 
\begin{equation*}
    F(p,X+Z)-F(p,X)=-\tr(A(p,t)Z) \leq -\a(|p|,t)\tr(Z). 
\end{equation*}
This becomes a special case of Part (5) of Proposition \ref{prop identify 1D}. 
\end{proof}
\section{Estimates with Neumann Boundary Conditions}\label{Sec Neumann}

The goal of this section is to show that the same modulus of continuity estimates for spatially periodic solutions in Theorem \ref{thm MC main} holds if $u$ solves \eqref{eq 1.1 fully nonlinear} 
on a bounded convex domain $\Omega \subset \R^n$ with smooth boundary and satisfies the Neumann boundary condition. 
The difference from the periodic case is that the local maximum point may lie on $\p (\Omega \times \Omega)$. For regular solutions of quasilinear equations, this possibility can be ruled out easily by assuming convexity of $\Omega$ as in \cite{AC09}. However, the same argument does not work for viscosity solutions because we cannot differentiate the equation. Moreover, the Neumann boundary condition needs to be understood in a weak sense as viscosity solutions are merely continuous.  

Let $\Omega \subset \R^n$ be a bounded domain with smooth boundary. We consider equations of the form
\begin{equation}\label{eq Boundary}
\begin{cases}
u_t+F(t,x, u, D u, D^2 u)=0, & \text{ in } \Omega \times (0,T), \\
B(t,x,u,Du,D^2 u) =0, & \text{ on } \p \Omega \times (0,T).
\end{cases}
\end{equation} 
Here both $F$ and $B$ are degenerate elliptic and continuous.
We recall the definition of viscosity solutions to the boundary value problem \eqref{eq Boundary} from \cite[Section 7]{CIL92}.

\begin{definition}
\begin{enumerate}
    \item  A function $u\in \USC(\ol{\Omega}\times (0,T))$ is a viscosity subsolution of  \eqref{eq Boundary} if
\begin{equation*}
     \tau+F(t,x,u(x,t),p,X) \leq 0
\end{equation*}
for  all $(x,t) \in \Omega \times(0,T),
(\tau, p ,X ) \in  \overline{\mathcal{P}}^{2,+}_{\ol{\Omega}\times (0,T)} u(x,t)$, and 
\begin{equation*}
    \min \left\{\tau+F(t,x,u(x,t),p,X), B(t,x, u(x,t),p,X) \right\} \leq 0
\end{equation*}
for all   $ (x,t) \in \p \Omega \times(0,T), (\tau, p ,X ) \in  \overline{\mathcal{P}}^{2,+}_{\ol{\Omega}\times (0,T)} u(x,t)$. 

\item A function $u\in \LSC(\ol{\Omega}\times (0,T))$ is a viscosity supersolution of  \eqref{eq Boundary} if
\begin{equation*}
    \tau+F(t,x,u(x,t),p,X) \geq 0
\end{equation*}
for all  $ (x,t) \in \Omega \times(0,T),
(\tau, p ,X ) \in  \overline{\mathcal{P}}^{2,-}_{\ol{\Omega}\times (0,T)} u(x,t)$,
and 
\begin{equation*}
    \max \left\{\tau+F(t,x,u(x,t),p,X), B(t,x,u(x,t),p,X) \right\} \geq 0
\end{equation*}
 for all  $  (x,t) \in \p \Omega \times(0,T), (\tau, p ,X ) \in  \overline{\mathcal{P}}^{2,-}_{\ol{\Omega}\times (0,T)} u(x,t).$

\item A viscosity solution of \eqref{eq Neumann} is a continuous function $u$ which is both a viscosity subsolution and a viscosity supersolution of \eqref{eq Neumann}.
\end{enumerate}
\end{definition}


The main result of this section is 
\begin{thm}\label{thm MC Neumann}
Let $\Omega \subset \R^n$ be a bounded convex domain with smooth boundary and diameter $D$.
Suppose $u$ is a viscosity solution of 
\begin{equation}\label{eq Neumann}
\begin{cases}
u_t+F(t,x, u, D u, D^2 u)=0, & \text{ in } \Omega \times (0,T), \\
\langle Du, \nu \rangle =0, & \text{ on } \p \Omega \times (0,T),
\end{cases}
\end{equation} 
where $\nu$ denotes the unit outward normal vector field along $\p \Omega$. 
Let $f(t,s,\vp,\vp',\vp'')$ be a one-dimensional operator satisfying \eqref{eq 1.3 structure}. 
Then the modulus of continuity $\omega$ of $u$ is a viscosity subsolution 
\begin{equation*}\label{}
    \w_t + f(t, s,\w, \w', \w'')=0
\end{equation*}
on $(0,D/2) \times (0,T)$ whenever $w$ is increasing.
\end{thm}

\begin{proof}
As in the proof of Theorem \ref{thm MC main}, we conclude that the function $Z:\overline{\Omega} \times \overline{\Omega} \times (0,T) \to \R$ defined by
$$Z(x,y,t):=u(x,t)-u(y,t)-2\vp\left(\frac{|x-y|}{2},t\right)$$
attains a local maximum zero at $(x_0,y_0,t_0) \in \overline{\Omega} \times \overline{\Omega} \times (0,T)$ with $|x_0-y_0|=2s_0$. 
If $(x_0,y_0)\in \Omega \times \Omega$, then the same argument as in the proof of Theorem \ref{thm MC main} would prove the theorem. So the strategy here is to perturb the equation to ensure that the maximum point always lies in the interior of $\Omega \times \Omega$. 
Note that $\w$ is increasing implies that $\vp'(s_0,t_0)\geq 0$. By approximation, we may assume that $\vp'(s_0,t_0) >0$. 

Pick a point $z_0\in \Omega$ and let $v(x)=\frac 1 2 (x-z_0)^2$. Then $Dv(x)=x-z_0$ and $D^2v(x)=I$.
Moreover for any $x\in \p \Omega$, 
$$\langle Dv(x), \nu(x) \rangle = \langle x-z_0, \nu(x) \rangle \geq d(z_0, \p \Omega) :=\delta >0. $$
Set
\begin{align*}
     u_\e(x,t) &= u(x,t)-\e v(x), \\
     u^\e(x,t) &= u(x,t)+\e v(x), \\
     F_\e(t,x,r,p,X)&=F(t,x,r+\e v,p+\e Dv, X+\e I),\\
     F^\e(t,x,r,p,X)&=F(t,x,r-\e v,p-\e Dv, X-\e I).
\end{align*}
Then direct calculation shows that $u_\e$ is a viscosity subsolution of 
\begin{equation}\label{}
\begin{cases}
u_t+F_\e(x, t, u, D u, D^2 u)=0, & \text{ in } \Omega \times (0,T), \\
\langle Du, \nu \rangle + \e \langle Dv, \nu \rangle =0, & \text{ on } \p \Omega \times (0,T),
\end{cases}
\end{equation} 
and 
$u^\e$ is a viscosity supersolution of 
\begin{equation}\label{}
\begin{cases}
u_t+F^\e(x, t, u, D u , D^2 u )=0, & \text{ in } \Omega \times (0,T), \\
\langle Du, \nu \rangle - \e \langle Dv, \nu \rangle =0, & \text{ on } \p \Omega \times (0,T). 
\end{cases}
\end{equation} 

We consider the following approximation of the function $Z$:
 $$Z_\e (x,y,t)=u_\e(x,t) -u^\e(y,t) -2 \vp\left(\frac{|x-y|}{2},t\right).$$
Since $Z_\e$ converges to $Z$ uniformly as $\e \to 0$, we know that $Z_\e$ has a local maximum at $(x_\e, y_\e , t_\e)$
with  $(x_\e, y_\e , t_\e) \to (x_0,y_0,t_0)$ and $2s_\e=|x_\e-y_\e| \to 2 s_0$ as $\e \to 0$.
By the parabolic maximum principle for semicontinuous functions (see Theorem \ref{thm max prin semicts}), for any $\lambda >0$, there exist $b_{1,\e }, b_{2,\e } \in \R$ and $X_\e,Y_\e\in S(n)$ such that 
\begin{align*}
    (b_{1,\e }, \vp'(s_\e, t_\e) e_\e, X_\e) & \in \overline{\mathcal{P}}^{2,+}_{\ol{\Omega}\times (0,T)} u_\e(x_\e,t_\e), \\
    (-b_{2,\e }, \vp'(s_\e, t_\e)e_\e, Y_\e) & \in \overline{\mathcal{P}}^{2,-}_{\ol{\Omega}\times (0,T)} u^\e(y_\e,t_\e),\\
    b_{1,\e }+b_{2,\e } &= 2\vp_t(s_\e, t_\e),\\
    -\left(\lambda^{-1}+\left\|M\right\| \right)I & \leq  \begin{pmatrix}
    X_\e & 0 \\
    0 & -Y_\e
    \end{pmatrix}
     \leq M+\lambda M^2,
\end{align*}
where $e_\e =\frac{x_\e -y_\e}{|x_\e -y_\e|}$ and  $M=D^2_{x,y}\left( 2 \vp \left(\frac{|x_\e-y_\e|}{2},t_\e \right) \right)$.

Since $u$ is a viscosity solution of \eqref{eq Neumann}, we have that
if $x_\e \in \Omega$, then at $(x_\e, t_\e )$,
\begin{equation}\label{x interior}
b_{1,\e} +F_\e(t_\e, x_\e, u_\e(x_\e,t_\e), \vp' e_\e , X_\e ) \leq 0,
\end{equation}
and if $x_\e \in \p \Omega$, then at $(x_\e ,t_\e )$, 
\begin{eqnarray*}
\min  \left\{b_{1,\e} +F_\e(t_\e, x_\e, u_\e(x_\e,t_\e), \vp' e_\e, X_\e ),  \vp^{\prime} \langle e_\e, \nu(x_\e)\rangle +\e \langle Dv(x_\e), \nu(x_\e) \rangle \right\} \leq 0,
\end{eqnarray*}

Since $\Omega$ is convex and $\vp'(s_\e,t_\e) \geq 0$ for $\e$ sufficiently small, we have
$\vp^{\prime} \langle e_\e, \nu(x_\e)\rangle +\e \langle Dv(x_\e), \nu(x_\e) \rangle \geq \e \delta >0 .$
Thus \eqref{x interior} is valid no matter $x_\e$ lies in $\Omega$ or on $\p \Omega$. 

Similarly,
if $y_\e \in \Omega$, then at $(y_\e ,t_\e )$,
 \begin{equation}\label{y interior}
 -b_{2,\e} +F^\e(t_\e, y_\e, u^\e(y_\e,t_\e), \vp' e_\e, Y_\e )  \geq 0,
 \end{equation}
and if $y_\e\in \p \Omega$, then at $(y_\e ,t_\e )$,
 \begin{eqnarray*}
 \max  \{-b_{2,\e}  +F^\e(t_\e, y_\e, u^\e(y_\e,t_\e), \vp' e_\e, -Y_\e ), \vp' \langle e_\e, \nu(y_\e)\rangle-\e \langle Dv(y_\e), \nu(y_\e) \rangle \} \geq 0.
 \end{eqnarray*}
Observe that $\vp' \langle e_\e, \nu(y_\e)\rangle-\e \langle Dv(y_\e), n(y_\e) \rangle \leq -\e \delta <0$, because $\Omega$ is convex and $\vp'\geq 0$.
Therefore,  \eqref{y interior}  is valid no matter $x_\e$ lies in $\Omega$ or on $\p \Omega$.
By passing to subsequences if necessary, we have
$b_{1,\e} \to b_1$, $b_{2,\e} \to b_2$, $X_\e  \to X$ and $Y_\e  \to Y$ as $\e  \to 0$. The limits satisfy
$$b_1+b_2=2\vp_t(s_0,t_0),$$
\begin{equation*}
   -\left(\lambda^{-1}+\left\|M\right\| \right)I \leq
     \begin{pmatrix}
    X & 0 \\
    0 & -Y
    \end{pmatrix}
    \leq M+\lambda M^2,
   \end{equation*}
where $M=D^2_{x,y}\left( 2 \vp \left(\frac{|x_0-y_0|}{2},t_0\right) \right)$. 
Letting $\e \to 0$ in \eqref{x interior} and \eqref{y interior} yields 
\begin{align*}
    b_1 +F(t_0,x_0,u(x_0,t_0),\vp' e_0, X) &\leq 0, \\
    -b_2+ F(t_0,y_0,u(y_0,t_0),\vp' e_0, Y) &\geq 0,
\end{align*}
where $e_0 =\frac{x_0-y_0}{|x_0-y_0|}$ and $\vp'$ is evaluated at $(s_0,t_0)$. 
The rest of the proof is exactly the same as the proof of Theorem 1.1.
\end{proof}


\section{Lipschitz Bounds and Gradient Estimates}\label{Sec Lip}

An obvious application of the modulus of continuity estimates is in proving regularity of solutions. It is well known that the heat equation evolves initial data which are very singular to solutions which are smooth for any positive time.  This is no longer true for more nonlinear equations even in one space dimension, particularly if the equation becomes degenerate when the gradient is large. Andrews and Clutterbuck \cite{AC09a}\cite{AC09} investigated the extent to which degenerate quasilinear parabolic equations smooth out irregular initial data. In particular, they gave a necessary and sufficient condition for quasilinear parabolic equations of one space variable to smooth our irregular initial data in \cite{AC09a} and  proved explicit gradient bounds for solutions of quasilinear equations such as graphical anisotropic mean curvature flows in \cite{AC09}. In this section, we further extend their results to fully nonlinear equations, as applications of the modulus of continuity estimates derived in previous sections.

As an immediate consequence of Theorem \ref{thm MC main} and \ref{thm MC Neumann}, we have 
\begin{prop}\label{prop Lipschitz}
Suppose that $u$ is either a spatially periodic viscosity solution of \eqref{eq 1.1 fully nonlinear} on $\R^n$ or a viscosity solution of \eqref{eq Neumann} on a bounded convex domain $\Omega \subset \R^n$ with smooth boundary.  
Let $f(t, s,\vp, \vp',\vp'')$ be a one-dimensional operator satisfying \eqref{eq 1.3 structure} and the comparison principle. 
Suppose that $\vp(s,t)$ satisfies 
\begin{enumerate}
    \item $\vp_t \geq f(t, s,\vp,\vp',\vp'')$ (in the viscosity sense);
    \item $\vp'(s,t) \geq 0$;
    \item $\vp(0,t) \geq 0$;
\end{enumerate}
for all $s$ and $t$.
If $\vp(s,0)$ is a modulus of continuity for $u(\cdot, 0)$, i.e., 
\begin{equation*}
    |u(x,0)-u(y,0)| \leq 2\vp \left(\frac{|x-y|}{2},0 \right)
\end{equation*}
for all $x, y$, then $\vp(s,t)$ is a modulus of continuity for $u(x,t)$, i,e., 
\begin{equation*}
    |u(x,t)-u(y,t)| \leq 2\vp \left(\frac{|x-y|}{2},t \right)
\end{equation*}
for all $x, y$.
\end{prop}

Here we say the one-dimensional operator $f(t,s,\vp,\vp',\vp'')$
satisfies the comparison principle if a subsolution is no bigger than a supersolution provided that this is true on the boundary and initially. This is satisfied by all examples given below. 

\begin{proof}[Proof of Proposition \ref{prop Lipschitz}]
For $\e>0$, the function $\vp_\e =\vp+\e e^t$ satisfies 
$$(\vp_\e)_t > f(t, s,\vp_\e, \vp_\e', \vp_\e''),$$ so it cannot touch the modulus of continuity $\w$ from above by Theorem \ref{thm MC main} or \ref{thm MC Neumann}. 
\end{proof}

Proposition \ref{prop Lipschitz} provides bounds on the modulus of continuity of solutions in terms of the initial modulus of continuity and elapsed time.
This then provides gradient estimates for $u$ at positive times, provided the particular solution of the one-dimensional equation has bounded gradient for positive time. Below we elaborate how to obtain such time-interior gradient estimates for fully nonlinear equations. 

For convenience, we introduce Assumption \eqref{F Ellipticity} which will be frequently used in this section. 
\begin{equation}\label{F Ellipticity}
     \nonumber
    \begin{cases}
    & F(t,r,p,X) \text{ is increasing in } r, 
    \text{ and } \\
    \tag{E} & F(t,r,p,X+Z) -F(t,r,p,X) \leq -\l(|p|) \tr(Z) \text{ for } Z \geq 0, \\ \nonumber
    & \text{where $\l(s)$ is  a nonnegative function }.
    \end{cases}
\end{equation}
The function $\l(s)$ measure the ellipticity of $F$. 
The time reparametrization $t \to ct $ for $c>0$ changes $F$ to $cF$, so it suffices to consider $\l(s)$ up to multiplying by a positive constant.

Let's first consider the heat equation $u_t=\Delta u$, which was discussed in \cite[Section 2]{Andrewssurvey15}. If the initial data $u_0$ is bounded, say $|u_0(x)| \leq M$, then $\vp_0(s)\equiv \frac{M}{2}$ is a modulus of continuity for $u_0$. It's easy to see that the one-dimension equation can be chosen to be the one-dimensional heat equation $\vp_t=\vp''$. With $\vp_0 \equiv \frac{M}{2}$ as initial data, the solution is given by 
\begin{equation*}
    \vp(s,t)=\frac{M}{2}\erf\left(\frac{s}{2\sqrt{t}} \right), 
\end{equation*}
where $\erf$ stands for the error function defined by
\begin{equation*}
    \erf(s)=\frac{2}{\sqrt{\pi}}\int_0^s e^{-t^2} dt.
\end{equation*}
By Proposition \ref{prop Lipschitz}, we have that for all $x$ and $y$  
\begin{equation*}
    |u(x,t)-u(y,t)| \leq M \erf\left(\frac{s}{4\sqrt{t}} \right).
\end{equation*}
Letting $y$ to $x$ then yields 
\begin{equation*}
    |Du(x,t)| \leq \frac{M}{2\sqrt{\pi t}},
\end{equation*}
for all $x$. 
Moreover, the estimate is sharp, with equality holding for the error function one-dimensional solution. 

It is remarkable that we can get the same conclusion for any fully nonlinear operator $F(t,r,p,X)$ that is uniformly elliptic, independent of $x$, and increasing in $u$.
\begin{prop}\label{prop gradient}
Let $u$ be as in Proposition \ref{prop Lipschitz}. Assume further that the operator $F(t,r,p,X)$ satisfies Assumption \eqref{F Ellipticity} with $\l(s) \equiv 1$.
If $|u(x,0)-u(y,0)|\leq M$ for some $M>0$, then 
\begin{equation*}
    |u(x,t)-u(y,t)| \leq M \erf \left(\frac{|x-y|}{4\sqrt{ t }} \right)
\end{equation*}
for all $x,y$ and $t \in (0,T)$.
In particular, 
\begin{equation*}
    |Du(x,t)| \leq \frac{M}{2\sqrt{\pi t}}
\end{equation*}
whenever it exists. 
\end{prop}

\begin{proof}
By Part (5) of Proposition \ref{prop identify 1D}, the one-dimensional equation can be chosen to be $\vp_t=\vp''$ since $F$ satisfies \eqref{F Ellipticity} with $\l(s) \equiv 1$. Its solution with $\vp(s,0)\equiv M/2$ is given by 
$$\vp(s,t)=\frac{M}{2}\erf\left(\frac{s}{2\sqrt{t}} \right).$$
The desired Lipschitz estimates then follows from Proposition \ref{prop Lipschitz} and the gradient estimates follows by letting $y \to x$. 
\end{proof}


The situation where the ellipticity $\l$ depends on $|p|$ is more complicated, even for quasilinear equations of one space variable \cite{AC09a}. But there are some cases where the one-dimensional equations can be solved and explicit gradient bounds can be obtained. Below we extract several examples from \cite{AC09} to further illustrate how the dependence of $\l$ on $|p|$ affect the time-interior gradient estimates.  

Let's consider an operator $F$ satisfying \eqref{F Ellipticity} with $\l(s)=(p-1)|s|^{p-2}$, where $1<p<\infty$.
Such $F$ could be $-(p-1)|Du|^{p-2}\Delta u$, the $p$-Laplacian $\Delta_p u :=-\text{div}(|Du|^{p-2}Du)$, or $-(p-1)|Du|^{p-2}M^+_{1,\Lambda}(D^2 u)$. 
The solution of the one-dimensional $p$-Laplacian heat flow $\vp_t =(p-1)|\vp'|^{p-2}\vp''$ with initial data $\vp_0 \equiv M/2$ is given by (see \cite[page 359]{AC09}) 
\begin{equation}\label{1D p}
    \vp(s,t)=\frac{M}{2}\frac{1}{2F_p(\infty)}F_p\left(\frac{s}{t^{1/p}R_p }\right),
\end{equation}
where 
\begin{equation*}
    F_p(z)=\begin{cases}
    \int_0^z (1+s^2)^{\frac{1}{p-2}} ds,  &  1<p<2; \\
    \int_0^z e^{-s^2} ds, & p=2;\\
    \int_0^z (1-s^2)_+^{\frac{1}{p-2}} ds, & p >2;
    \end{cases}
\end{equation*}
$F_p(\infty) =\lim_{z \to \infty} F_p(z)$, and 
\begin{equation*}
    R_p=\begin{cases}
    \left(\frac{2-p}{2p(p-1)} \right)^{\frac{-1}{p}} \left(2F_p(\infty) \right)^{\frac{2-p}{p}}, & 1<p<2;\\
    2, & p=2;\\
    \left(\frac{2p(p-1)}{p-2}\right)^{\frac{1}{p}} \left(2F_p(\infty) \right)^{\frac{2-p)}{p}}, & p>2.
    \end{cases}
\end{equation*}

Thus we obtain with Theorem \ref{thm MC main} and Theorem \ref{thm MC Neumann} that 
\begin{prop}\label{prop gradient p}
Let $u$ be as in Proposition \ref{prop Lipschitz}. Assume further that the operator $F(t,r,p,X)$ satisfies Assumption \eqref{F Ellipticity} with $\l(s) \equiv (p-1)|s|^{p-2}$, where $1<p<\infty$.
If $|u(x,0)-u(y,0)|\leq M$ for some $M>0$, then 
\begin{equation*}
    |u(x,t)-u(y,t)| \leq 2 \vp \left(\frac{|x-y|}{2}, t \right)
\end{equation*}
for all $x,y$ and $t \in (0,T)$, where $\vp$ is given by \eqref{1D p}. 
In particular, 
\begin{equation*}
    |Du(x,t)| \leq \frac{1}{2R_pF_p(\infty)} \frac{M^{\frac{2}{p}}}{t^{\frac{1}{p}}}
\end{equation*}
whenever it exists.
\end{prop}

Another example is when $F$ satisfies Assumption \eqref{F Ellipticity} with $\l(s)=\frac{1}{1+s^2}$. Such $F$ could be $\frac{1}{1+|Du|^2} \Delta u$ or $\frac{1}{1+|Du|^2}M^+_{1,\Lambda}(D^2 u)$. 
In this case, one can use the results in 
\cite[Theorem 2.1]{AC09} to conclude that if $|u_0| \leq M$, then 
\begin{equation*}
    1+|Du(x,t)|^2 \leq \exp\left(\frac{2M^2}{t} \right)
\end{equation*}
whenever it exists. 
Similarly as in \cite[Corollary 3.2]{AC09}, if $F$ satisfies Assumption \eqref{F Ellipticity} with $\l(s, t)$ satisfying that there exist positive constants $A_0$ and $P_0$ such that 
\begin{equation*}
    \l(s, t) \geq \frac{A_0}{s^2} \text{ for } s \geq P_0  \text{ and all } t, 
\end{equation*}
then 
\begin{equation*}
    |Du(x,t)| \leq P_0 \exp \left(1+\frac{M^2}{A_0 t} \right) 
\end{equation*}
whenever it exists. Such equations include the anisotropic mean curvature flows. 



Finally, under some condition on $\l(|p|)$, we can use the results in \cite{AC09a} to get gradient bounds for any positive time if the initial data is bounded. 


\begin{prop}
Let $u$ be as in Proposition \ref{prop Lipschitz}. Assume further that the operator $F(t,r,p,X)$ satisfies Assumption \eqref{F Ellipticity} with $\l(s)$ satisfying $\int_0^a s \l(s) ds \to \infty$ as $a \to \infty$. 
If $|u(x,0)-u(y,0)|\leq M$ for some $M>0$ and $u$ is $C^1$, then there exist a constant $C$ depending only on $M$ and $t$ such that 
\begin{equation*}
    |Du(x,t)| \leq C 
\end{equation*}
for all $x,y$ and $t \in (0,T)$.
\end{prop}

\begin{proof}
By part (3) of Proposition \ref{prop identify 1D}, the one-dimensional equation can be chosen to be $\vp_t=\l(|\vp'|)\vp''$. It was shown in \cite[Section 5]{AC09a} that $\vp$ has its gradient bounded for all positive times by a constant depending only on $t$ and the initial oscillation if and only if $\lim_{s\to \infty} B(s)=\infty$, where $B(s)=\int_0^s t \l (t) dt$. Thus we have the desired gradient bounds. 
\end{proof}

\section{Extensions to Riemannian Manifolds}

In this section, we extend the modulus of continuity estimates to fully nonlinear parabolic equations on Riemannian manifolds. This is mostly a straightforward matter, but requires some adaptions and modifications to overcome the non-smoothness of the Riemannian distance function. For the basic theory of viscosity solutions on manifolds, see \cite{AFS08}\cite{I}\cite{Zhu11}.

Let $(M^n,g)$ be an $n$-dimensional complete Riemannian manifold without boundary. The Riemannian distance function $d(x,y)$ is given by 
\begin{equation*}
    d(x,y)=\inf \{L[\gamma]: \text{ $\gamma$ is a smooth path from $x$ to $y$} \}.
\end{equation*}
where $L[\gamma]$ stands for the length of $\gamma$. 
The modulus of continuity $\w$ of a function $u:M \to \R$ is defined similarly as before by 
\begin{equation*}
    \w(s):=\sup \left\{\frac{u(x)-u(y)}{2} : x,y \in M, d(x,y)=2s \right\}.
\end{equation*}

As in the Euclidean case, we consider parabolic equations of the form 
\begin{equation}\label{}
   u_t+F(t,x,u,Du,D^2 u)=0,
\end{equation}
where $F:[0,T] \times M \times \R \times TM \times \text{Sym}^2T^*M \to \R$ is degenerate elliptic and continuous, $Du$ and $D^2 u$ denote the gradient and Hessian of $u$ with respect to the Levi-Civita connection, and $TM$, $T^*M$ and $\text{Sym}^2T^*M$ denote the tangent bundle, the cotangent bundle and the set of symmetric two tensors on $M$, respectively. 

We would like to introduce a structure condition that is analogous to \eqref{eq 1.3 structure} in the Euclidean setting. One simply needs to replace $|x-y|$ in \eqref{eq 1.3 structure} by $d(x,y)$, i.e.,  
\begin{equation}\label{eq 6.3 structure mfd}
    \begin{cases}
    & F\left(t,y,r,-D_y\psi, Y\right) -F\left(t,x,v, D_x\psi, X\right) \leq -2f(t, s, \vp,\vp',\vp'') \\ 
   \tag{SCM} & \text{ for all } 
    x, y \in M \text{ with } d(x,y)=2s>0,  v,r \in \R \text{ with } v-r=\psi >0, \\ \nonumber
    & \text{ and } X\in \text{Sym}^2T_x^*M, Y\in \text{Sym}^2T_y^*M\text{ satisfying } 
     \begin{pmatrix}
    X & 0 \\ 0 & -Y 
    \end{pmatrix}
    \leq D^2_{x,y} \psi(x,y,t),
    \end{cases}
\end{equation}
where $\psi(x,y,t)=2\vp\left(\frac{d(x,y)}{2},t \right)$ and all derivatives of $\psi$ are evaluated at $(x,y,t)$. 
Note that $d(x,y)$ is in general not a smooth function, so \eqref{eq 6.3 structure mfd} should be understood in the ``viscosity sense", i.e., if $d$ is not smooth at $(x,y)$, then \eqref{eq 6.3 structure mfd} holds with 
$$\psi(x,y,t)=  2\vp\left(\frac{\rho(x,y)}{2},t \right) $$
for any smooth function $\rho$ touching $d$ from above at $(x,y)$. 
Finally, we point out that curvatures effect the one-dimensional operator $f$ as the Hessian of the distance function depends on curvatures.  






We have the modulus of continuity estimates on manifolds. 
\begin{thm}\label{thm MC mfd}
Let $(M^n,g)$ be a closed Riemannian manifold of dimension $n\geq 2$ and diameter $D$. Suppose that $u:M \times [0,T) \to \R$ is a viscosity solution of 
\begin{equation}\label{eq 6.1 fully nonlinear mfd}
    u_t +F(t,x,u,Du,D^2 u)=0.
\end{equation}
Let $f(t,s,\vp,\vp',\vp'')$ be a one-dimensional operator satisfying \eqref{eq 6.3 structure mfd}. 
Then the modulus of continuity $\w$ of $u$ is a viscosity subsolution of 
\begin{equation*}
    \vp_t = f(t,s,\w,\w',\w'')
\end{equation*}
on $(0,D/2)\times (0,T)$ whenever $\w$ is increasing in $s$.
\end{thm}

\begin{proof}
The proof is a slight modification of that of Theorem \ref{thm MC main}. 
Let $\vp$ be a smooth function touching $w$ from above at $(s_0, t_0) \in (0,D/2)\times (0,T)$.
The assumption that $\omega$ is increasing in $s$ implies $\vp'(s_0,t_0) \geq 0$.
Since $M$ is compact, there exist $x_0$ and $y_0$ in $M$ with $d(x_0, y_0)=2s_0$ such that
$$u(x_0, t_0)-u(y_0, t_0)=2w(s_0,t_0)=2\vp(s_0,t_0).$$
For $x$ close to $x_0$, $y$ close to $y_0$, and $t$ close to $t_0$, it holds that 
\begin{equation*}
    u(x,t)-u(y,t) \leq 2\w\left(\frac{d(x,y)}{2},t\right) \leq 2\vp\left(\frac{d(x,y)}{2},t\right).
\end{equation*}
Thus the function
$$u(x,t)-u(y,t)-2\vp\left(\frac{d(x,y)}{2},t\right) $$
attains a local maximum at $(x_0,y_0,t_0)$. 

Now let $\rho$ be any smooth function satisfying  $d(x,y)\le \rho(x, y)$ near $(x_0,y_0)$ with equality at $(x_0, y_0)$. Since $\vp$ is non-decreasing, the function
$$Z(x, y, t):=u(x,t)-u(y, t)-\psi(x,y,t)$$
has a local maximum zero at $(x_0,y_0,t_0)$, where $\psi(x,y,t)=2\vp\left(\frac{\rho(x,y)}{2},t\right)$. 

By the parabolic maximum principle for semicontinuous functions on manifolds (see \cite[Section 2.2]{I} or \cite[Theorem 2.3]{LW17}), we have that for each $\lambda >0$, there exist $b_1, b_2 \in \R$, and $X \in \text{Sym}^2T_{x_0}^*M, Y \in \text{Sym}^2T_{y_0}^*M$ such that
\begin{align*}
    (b_1, D_x \psi, X) & \in \overline{\mathcal{P}}^{2,+} u(x_0,t_0),\\
    (-b_2, - D_y \psi, Y) & \in \overline{\mathcal{P}}^{2,-} u(y_0,t_0), \\
     b_1+b_2 & = \psi_t =2 \vp_t(s_0,t_0), \\
     \begin{pmatrix}
   X & 0 \\
   0 & -Y
   \end{pmatrix}
   & \leq M+\lambda M^2,
\end{align*}
where $M=D^2 \psi$. Here are below, all derivatives of $\psi$ are evaluated at $(x_0,y_0,t_0)$ and all derivatives of $\vp$ are evaluated at $(s_0,t_0)$. 

Since $u$ is a viscosity solution of \eqref{eq 6.1 fully nonlinear mfd}, we have 
\begin{align*}
    b_1 +F(t_0,x_0,u(x_0,t_0), D_x\psi, X) \leq 0, \\
    -b_2 +F(t_0,y_0,u(y_0,t_0),-D_y\psi, Y) \geq 0.
\end{align*}
Therefore, we obtain by letting $\l \to 0^+$ that
\begin{align*}
    2\vp_t=b_1+b_2 & \leq F(t_0,y_0,u(y_0,t_0),-D_y\psi, Y) -F(t_0,x_0,u(x_0,t_0), D_x\psi, X) \\
    & \leq -2f(t_0, s_0,\vp,\vp',\vp''),
\end{align*}
where we have used the structure condition \eqref{eq 6.3 structure mfd} in the last inequality. This completes the proof. 
\end{proof}

\section{Effects of Curvatures on One-dimensional Operators}

On curved spaces, the Hessian of the distance function depends on sectional curvatures. The well known Hessian and Laplacian comparison theorems (see for example \cite{Sakai96}) provide sharp comparison with distance functions on spaces of constant sectional curvature $\k \in \R$, which are spheres ($\k >0$), Enclidean spaces ($\k=0$), and hyperbolic spaces ($\k <0$). To control the full Hessian of the distance function, one needs bounds on the sectional curvatures. However, if we only need to estimate the  Laplacian of the distance function from above, then lower bounds on Ricci curvature suffice.

We extract an example from \cite{AC13} to show how the curvatures effect the one-dimensional operators.


\begin{prop}\label{prop identify 1D mfd}
Assume the Ricci curvature of $M$ is bounded from below by $(n-1)\k$ for some $\k\in \R$. Then the following pair of $f$ and $F$ satisfy the structure condition \eqref{eq 6.3 structure mfd}. 
    \begin{equation*}\label{}
        F(p,X)=-\tr \left[ \left(\a(|p|)\frac{p \otimes p}{|p|^2} +\b(|p|) \left(I- \frac{p \otimes p}{|p|^2}\right) \right) X \right],
\end{equation*}
where $\a$ and $\b$ are nonnegative functions. 
    \begin{equation*}
        f(\vp',\vp'')=-\a(\vp')\vp''+(n-1)\b(\vp')T_{\k},
    \end{equation*}
where the function $T_k$ is defined for $\k \in \R$ by
\begin{equation}\label{def T}
     T_\k(t)=\begin{cases}
   \sqrt{\k} \tan{(\sqrt{\k}t)}, & \k>0, \\
   0, & \k=0, \\
   -\sqrt{-\k}\tanh{(\sqrt{-\k}t)}, & \k<0.
    \end{cases}
\end{equation}
\end{prop}


\begin{proof}[Proof of Proposition \ref{prop identify 1D mfd}]
The proof is essentially the same as in \cite{AC13}. 
Let $x,y\in M$ be such that $d(x,y)=2s>0$ and $\gamma:[0,1] \to M$ be a length-minimizing geodesic connecting $x$ and $y$ with $|\gamma'|=2s$. Choose an orthonormal frame $\{e_i(0)\}_{i=1}^n$ for $T_{x_0}M$ with $e_n(0)=\gamma'(0)$ and parallel translate it along $\gamma$ to produce orthonormal frame $\{e_i(s)\}_{i=1}^n$ for $T_{\gamma(s)}M$ with $e_n(s)=\gamma'(s)$ for all $s\in [0,1]$. 

Let $\rho$ be a smooth function touching $d$ from above at $(x,y)$ and write $\psi(x,y,t)=2\vp\left(\frac{\rho(x,y)}{2},t \right)$. Direct calculation shows that
\begin{align*}
    D_x\psi(x,y,t) &=\vp'(s,t)\gamma'(0), \\
    D_y\psi(x,y,t) &=-\vp'(s,t)\gamma'(1).
\end{align*}
Therefore, we have 
\begin{eqnarray*}
&& F(-D_y\psi, Y)-F(D_x \psi, X) \\
&=&  -\a(\vp') Y(e_n(1),e_n(1)) -\b(\vp')\sum_{i=1}^{n-1}Y(e_i(1),e_i(1)) \\
&& +\a(\vp') X(e_n(0),e_n(0)) +\b(\vp')\sum_{i=1}^{n-1}X(e_i(0),e_i(0)) \\
&=& \a(\vp') \begin{pmatrix}
   X & 0 \\
   0 & -Y
   \end{pmatrix} \left((e_n(0),-e_n(1)), (e_n(0),-e_n(1)) \right) \\
  && +\b(\vp') \sum_{i=1}^{n-1}\begin{pmatrix}
   X & 0 \\
   0 & -Y
   \end{pmatrix} \left((e_i(0),e_i(1)), (e_i(0),e_1(1)) \right)\\
   &\leq& \a(\vp') D^2 \psi \left((e_n(0),-e_n(1)), (e_n(0),-e_n(1)) \right) \\
  && +\b(\vp') \sum_{i=1}^{n-1} D^2 \psi \left((e_i(0),e_i(1)), (e_i(0),e_1(1)) \right).\\
\end{eqnarray*}
It is easy to calculate that 
\begin{eqnarray*}
&& D^2 \psi \left((e_n(0),-e_n(1)), (e_n(0),-e_n(1)) \right) \\
&=& \vp'' D\rho \otimes D\rho \left((e_n(0),-e_n(1)), (e_n(0),-e_n(1)) \right) \\
&& +\vp' D^2 \rho \left((e_n(0),-e_n(1)), (e_n(0),-e_n(1)) \right) \\
&=& 2\vp''
\end{eqnarray*}
and 
\begin{eqnarray*}
&& \sum_{i=1}^{n-1} D^2 \psi \left((e_i(0),e_i(1)), (e_i(0),e_1(1)) \right) \\
&=&\vp'' \sum_{i=1}^{n-1} D \rho \otimes D \rho \left((e_i(0),-e_i(1)), (e_i(0),-e_i(1)) \right) \\
&& +\vp'\sum_{i=1}^{n-1}  D^2 \rho \left((e_i(0),-e_i(1)), (e_i(0),-e_i(1)) \right) \\
&\leq& -2 (n-1)T_\k \vp',
\end{eqnarray*}
where we have used the Laplacian comparison theorem in the last inequality (see for example \cite[page 1018]{AC13} or \cite[page 564]{LW17}). 
Combining the above estimates, we obtain 
\begin{equation*}
    F(-D_y\psi, Y)-F(D_x \psi, X) \leq -2\left(\a(\vp')\vp'' -(n-1)T_\k \vp' \b(\vp') \right).
\end{equation*}
Thus $f= -\a(\vp')\vp'' +(n-1)T_\k \vp' \b(\vp')$ satisfies \eqref{eq 6.3 structure mfd}. 

\end{proof}

Analogous results hold for other examples given in Proposition \ref{prop identify 1D} and \ref{prop 1D identify} (Part (4) of Proposition \ref{prop identify 1D} requires nonnegative sectional curvature). As applications, we obtain the same Lipschitz bounds and gradient estimates as in Section \ref{Sec Lip} provided that $M$ has nonnegative Ricci curvature.


\bibliographystyle{plain}
\bibliography{ref}

\begin{thebibliography}{10}

\bibitem{Andrewssurvey12}
Ben Andrews.
\newblock Gradient and oscillation estimates and their applications in
  geometric {PDE}.
\newblock In {\em Fifth {I}nternational {C}ongress of {C}hinese
  {M}athematicians. {P}art 1, 2}, volume~2 of {\em AMS/IP Stud. Adv. Math., 51,
  pt. 1}, pages 3--19. Amer. Math. Soc., Providence, RI, 2012.

\bibitem{Andrewssurvey15}
Ben Andrews.
\newblock Moduli of continuity, isoperimetric profiles, and multi-point
  estimates in geometric heat equations.
\newblock In {\em Surveys in differential geometry 2014. {R}egularity and
  evolution of nonlinear equations}, volume~19 of {\em Surv. Differ. Geom.},
  pages 1--47. Int. Press, Somerville, MA, 2015.

\bibitem{AC09a}
Ben Andrews and Julie Clutterbuck.
\newblock Lipschitz bounds for solutions of quasilinear parabolic equations in
  one space variable.
\newblock {\em J. Differential Equations}, 246(11):4268--4283, 2009.

\bibitem{AC09}
Ben Andrews and Julie Clutterbuck.
\newblock Time-interior gradient estimates for quasilinear parabolic equations.
\newblock {\em Indiana Univ. Math. J.}, 58(1):351--380, 2009.

\bibitem{AC11}
Ben Andrews and Julie Clutterbuck.
\newblock Proof of the fundamental gap conjecture.
\newblock {\em J. Amer. Math. Soc.}, 24(3):899--916, 2011.

\bibitem{AC13}
Ben Andrews and Julie Clutterbuck.
\newblock Sharp modulus of continuity for parabolic equations on manifolds and
  lower bounds for the first eigenvalue.
\newblock {\em Anal. PDE}, 6(5):1013--1024, 2013.

\bibitem{AN12}
Ben Andrews and Lei Ni.
\newblock Eigenvalue comparison on {B}akry-{E}mery manifolds.
\newblock {\em Comm. Partial Differential Equations}, 37(11):2081--2092, 2012.

\bibitem{AFS08}
Daniel Azagra, Juan Ferrera, and Beatriz Sanz.
\newblock Viscosity solutions to second order partial differential equations on
  {R}iemannian manifolds.
\newblock {\em J. Differential Equations}, 245(2):307--336, 2008.

\bibitem{Balch20}
Le~K\'evin Balc'h.
\newblock Exponential bounds for gradient of solutions to linear elliptic and
  parabolic equations.
\newblock {\em arXiv:2006.04582}, 2020.

\bibitem{Clutterbuck07}
Julie Clutterbuck.
\newblock Interior gradient estimates for anisotropic mean-curvature flow.
\newblock {\em Pacific J. Math.}, 229(1):119--136, 2007.

\bibitem{CIL92}
Michael~G. Crandall, Hitoshi Ishii, and Pierre-Louis Lions.
\newblock User's guide to viscosity solutions of second order partial
  differential equations.
\newblock {\em Bull. Amer. Math. Soc. (N.S.)}, 27(1):1--67, 1992.

\bibitem{DSW18}
Xianzhe Dai, Shoo Seto, and Guofang Wei.
\newblock Fundamental gap estimate for convex domains on sphere -- the case
  n=2.
\newblock {\em Comm. Anal. Geom., to appear, arXiv:1803.01115}, 2018.

\bibitem{HW17}
Chenxu He and Guofang Wei.
\newblock Fundamental gap of convex domains in the spheres (with appendix {B}
  by {Q}i {S}. {Z}hang).
\newblock {\em Amer. J. Math., to appear, arXiv:1705.11152}, 2017.

\bibitem{I}
Tom Ilmanen.
\newblock Generalized flow of sets by mean curvature on a manifold.
\newblock {\em Indiana Univ. Math. J.}, 41(3):671--705, 1992.

\bibitem{Kruzhkov67}
S.~N. Kru\v{z}kov.
\newblock Nonlinear parabolic equations with two independent variables.
\newblock {\em Trudy Moskov. Mat. Ob\v{s}\v{c}.}, 16:329--346, 1967.

\bibitem{Li16}
Xiaolong Li.
\newblock Moduli of continuity for viscosity solutions.
\newblock {\em Proc. Amer. Math. Soc.}, 144(4):1717--1724, 2016.

\bibitem{LW17}
Xiaolong Li and Kui Wang.
\newblock Moduli of continuity for viscosity solutions on manifolds.
\newblock {\em J. Geom. Anal.}, 27(1):557--576, 2017.

\bibitem{LW19eigenvalue}
Xiaolong Li and Kui Wang.
\newblock Sharp lower bound for the first eigenvalue of the weighted
  $p$-laplacian.
\newblock {\em arXiv:1910.02295}, 2019.

\bibitem{LW19eigenvalue2}
Xiaolong Li and Kui Wang.
\newblock Sharp lower bound for the first eigenvalue of the weighted
  $p$-laplacian {II}.
\newblock {\em Math. Res. Lett, to appear, arXiv:1911.04596}, 2019.

\bibitem{Ni13}
Lei Ni.
\newblock Estimates on the modulus of expansion for vector fields solving
  nonlinear equations.
\newblock {\em J. Math. Pures Appl. (9)}, 99(1):1--16, 2013.

\bibitem{Sakai96}
Takashi Sakai.
\newblock {\em Riemannian geometry}, volume 149 of {\em Translations of
  Mathematical Monographs}.
\newblock American Mathematical Society, Providence, RI, 1996.
\newblock Translated from the 1992 Japanese original by the author.

\bibitem{SWW19}
Shoo Seto, Lili Wang, and Guofang Wei.
\newblock Sharp fundamental gap estimate on convex domains of sphere.
\newblock {\em J. Differential Geom.}, 112(2):347--389, 2019.

\bibitem{Zhu11}
Xuehong Zhu.
\newblock Viscosity solutions to second order parabolic {PDE}s on {R}iemannian
  manifolds.
\newblock {\em Acta Appl. Math.}, 115(3):279--290, 2011.

\end{thebibliography}

\end{document}